\documentclass[10pt,a4paper]{article}
\usepackage{amsmath}
\usepackage{amssymb}
\usepackage{amsbsy}
\usepackage{amsfonts}
\usepackage{graphicx}
\usepackage{sectsty}
\usepackage{verbatim}
\usepackage{enumitem}
\usepackage{hyperref}
\usepackage{ifthen}

\usepackage{color}

\numberwithin{equation}{section}
\newtheorem{theorem}{Theorem}[section]

\newtheorem{lemma}{Lemma}[section]
\newtheorem{definition}{Definition}[section]

\newtheorem{remark}{Remark}[section]
\newtheorem{algorithm}{Algorithm}[section]
\newenvironment{proof}{{\textbf{Proof.}}}{\hfill \textbf{$\square$}\vspace{0.2cm}}

\newcommand{\abs}[1]{\lvert#1\rvert}
\newcommand{\norm}[1]{\lVert#1\rVert}

\newcommand{\veps}[1]{#1_{\varepsilon}}
\newcommand{\ve}{\varepsilon}
\newcommand{\vp}[2]{\dfrac{\partial #1}{\partial #2}}

\def\yt{\hat{y}}
\def\uo{u}
\def\ut{u^{(\yt,\ve)}}
\def\vo{v}
\def\vt{v^{(\yt,\ve)}}
\def\Bo{K}
\def\Bt{B_{\ve}(\yt)}

\def\omegt{\Omega\setminus(\bar{K}\cup\bar{B}_{\ve}(\hat{y}))}

\newcommand{\SBV}{{\mathit SBV}}

\DeclareMathOperator{\dist}{dist}
\DeclareMathOperator*{\argmin}{arg\,min}
\newcommand{\Gammalim}{\Gamma\text{-}\lim}
\newcommand{\Gammaliminf}{\Gamma\text{-}\liminf}
\newcommand{\Gammalimsup}{\Gamma\text{-}\limsup}

\newcommand{\field}[1]{\ensuremath{\mathbb{#1}}}
\newcommand{\R}{\field{R}}
\newcommand{\N}{\field{N}}

\newcommand{\Sp}{\mathbb{S}}

\newcommand{\X}{X}
\newcommand{\U}{U}

\makeatletter
\newcommand{\logmessage}[1]{\@latex@warning{#1}}
\newcommand{\ignore}{\logmessage{Text ignored}\@gobble}
\makeatother

\begin{document}

\vspace{0.2cm}
\begin{center}
{\textbf{An Approach to the Minimization of the Mumford--Shah Functional
using $\Gamma$-convergence and Topological Asymptotic Expansion}} \\[0.5cm]
{\small Markus Grasmair,$^{1,2}$ Monika Muszkieta,$^3$ and Otmar Scherzer$^{1,4}$\\[0.2cm]
${}^1$Computational Science Center\\
University of Vienna\\
Nordbergstr.~15, 1090 Wien, Austria\\[0.2cm]
${}^2$Faculty for Mathematics and Geography\\
Catholic University Eichst\"att--Ingolstadt\\
Ostenstr.~26, 85072 Eichst\"att, Germany\\[0.2cm]
${}^3$Institute of Mathematics and Computer Science \\
Wroclaw University of Technology \\
ul. Wybrzeze Wyspianskiego 27, 50-370 Wroclaw, Poland \\[0.2cm]
${}^4$RICAM\\
Austrian Academy of Sciences\\
Altenbergerstr.~69, 4040 Linz, Austria\\[0.2cm]
markus.grasmair@univie.ac.at \\
monika.muszkieta@pwr.wroc.pl \\
otmar.scherzer@univie.ac.at\\[0.2cm]
\today
}
\end{center}
\vspace{0.3cm}

\begin{abstract}
In this paper, we present a method for the numerical
minimization of the Mumford--Shah functional that is based on
the idea of topological asymptotic expansions.
The basic idea is to cover the expected edge set
with balls of radius $\ve > 0$ and use the number of balls, multiplied with
$2\ve$, as an estimate for the length of the edge set.
We introduce a functional based on this idea and prove that it converges
in the sense of $\Gamma$-limits to the Mumford--Shah functional.
Moreover, we show that ideas from topological asymptotic analysis
can be used for determining where to position the balls covering the edge set.
The results of the proposed method are presented by means of two numerical examples
and compared with the results of the classical approximation due to Ambrosio and Tortorelli.

\end{abstract}

\vspace{0.2cm}

\noindent
{\bf Keywords:} Topological asymptotic expansion; $\Gamma$-convergence; Mumford-Shah functional; image segmentation.\\[0.2cm]
{\bf AMS Classification:} 35R35, 65K10, 49M25.


\section{Introduction}

Let $\Omega$ be a Lipschitz bounded open domain in $\mathbb{R}^2$.
We assume that a possibly noisy image on $\Omega$ is given,
represented by a real-valued, bounded function $f$ on $\Omega$,
whose values $f(x)$, $x \in \Omega$, correspond to the intensity of $f$ at the pixel $x$.
In order to segment, and denoise at the same time, the image $f$,
Mumford and Shah~\cite{MumSha89} introduced a variational model,
which is based on the assumptions that the different ``objects''
in the image give rise to homogeneous regions that are separated
by the objects' projected silhouettes.
Moreover, these silhouettes in general correspond to discontinuities in the image $f$.
By this reasoning, they proposed to minimize the functional
\begin{equation}\label{functional_MS}
  F(u,K) = \frac{1}{2}\int_\Omega (u-f)^2\,dx
  + \frac{\alpha}{2}\int_{\Omega\setminus K} \lvert\nabla u\rvert^2\,dx
  + \beta\mathcal{H}^1(K)\,,
\end{equation}
taking as variables the function $u \in L^2(\Omega)$,
the denoised image, and the compact set $K \subset \Omega$,
the set of edges or silhouettes.
Here, the value $F(u,K)$ is set to $+\infty$ if the restriction of
  the weak derivative $\nabla u$ of $u$ to $\Omega\setminus K$ is not square-integrable.
Here, $\mathcal{H}^1(K)$ denotes the one-dimensional Hausdorff measure
of the set $K$; in the case $K$ is a regular (rectifiable) one-dimensional set,
this is precisely its length.
The parameters $\alpha$ and $\beta$ that appear in \eqref{functional_MS} are positive constants
determining the weight that is put on the regularity of
the denoised image $u$ and the length of the edge set $K$.

In order to show the existence of minimizers of the Mumford--Shah functional,
a weak formulation depending only on one variable has been introduced by
De Giorgi et al.~\cite{DeGCarLea89}.
In their model, the set $K$ is replaced by the jump set $S_u$ of the function $u \in SBV(\Omega)$,
the space of special function of bounded variation on $\Omega$.
Still, this reformulation of the functional provides no method
for the actual numerical computation of minimizers.
Thus, various approximations of the functional $F$ have been proposed,
most of them based on the concept of $\Gamma$-convergence.
Ambrosio and Tortorelli \cite{AmbTor90} proposed a variational model in which they replace the set $S_u$
by a continuous function $v$ with values close to $0$ near $S_u$, and values close to $1$ away from $S_u$.
For other approximations of the Mumford--Shah functional in the sense of $\Gamma$-convergence,
we refer to Braides et al.~\cite{BraDal97}, where the authors propose approximations by a family of non-local functionals.
Approximations by finite-difference schemes,
inspired by the original discrete model of Blake and Zisserman~\cite{Blake1987},
have been considered by Chambolle~\cite{Chambolle1995,Chambolle1999a},
and by finite-elements schemes by Chambolle and Dal Maso~\cite{Chambolle1999b}
or recently by Aubert et al.~\cite{AubBlaMar06}.
In the work by Koepfler et al.~\cite{Koepfler1994} and Dal Maso et al.~\cite{DalMaso1992},
region growing and merging methods have been proposed.
For detailed analysis of the Mumford--Shah model,
we refer reader to the book by Morel and Solimini \cite{Morel1995}.
We also refer to the books \cite{Aubert2006,Scherzer2010,SchGraGroHalLen09},
where some of the above mentioned results are shortly discussed.

In the following, we will show how the problem of minimizing the
Mumford--Shah functional $F$ can be approached using ideas of
topological asymptotic analysis.
In its original formulation (see, e.g., Soko{\l}owski et al.~\cite{Sok99},
Garreau et al.~\cite{Gar00}, Feij\'oo et al.~\cite{Fei03}),
this theory investigates a variation of a given objective functional
depending on some domain in $\R^n$ with respect to
the subtraction of a small ball from this domain.
This variation is a scalar function, called the topological gradient or the topological derivative,
and its largest negative values indicate positions, where it is good to remove a small ball.
In \cite{Ams06}, Amstutz proposed to modify the definition of
topological gradient and provide the variation of a given functional with
respect to change of certain material properties, and not a domain
topology. The topological derivative of an objective functional
  has also been considered by Giusti et
  al.~\cite{Giusti08} in the case
  of nonlinear elasticity.
Recently, topological asymptotic analysis has been also applied by
Auroux et al.~\cite{Aur07,Aur09} and by Muszkieta \cite{Mus09} to
various problems in image processing.


In order to apply the theory of topological asymptotic expansions,
it is necessary that the functional to be minimized
depends solely on the set $K$.
This can be achieved in the case of the Mumford--Shah functional
by noting that the minimizing pair $(u_0,K_0)$ is uniquely determined
by either of the components $u_0$ and $K_0$:
The set $K_0$ coincides with the jump set $S_{u_0}$ of the function $u_0$.
Conversely, $u_0$ can be computed from $K_0$ by solving the
partial differential equation
\begin{equation}\label{eq:MS_PDE}
  \begin{aligned}
    u - \alpha \Delta u & = f && \text{ in } \Omega\,,\\
    \frac{\partial u}{\partial n} & = 0 &&\text{ in } \partial (\Omega \setminus K_0)\;.
  \end{aligned}
\end{equation}
Now consider the functional
\[
J(K) = F\bigl(u(K),K\bigr)\,,
\]
where $u(K)$ denotes the solution of~\eqref{eq:MS_PDE}
with $K_0$ replaced by $K$.
Then the pair $(u(K_0),K_0)$ minimizes the functional $F$,
if and only if $J(K_0) \le J(K)$ for all $K \subset \Omega$.

The idea is now to use a
gradient descent like approach to the minimization of the functional $J$.
Starting from an initial guess $K$ of the edge set
(for instance $K = \emptyset$),
one adds to $K$ those points, whose inclusion would
lead to a near to maximal decrease of the cost functional $J$.
More precisely, one adjoins to the set $K$ small balls of radius $\ve > 0$
centered at the points $x \in \Omega \setminus K$ and
tries to compute an asymptotic expansion of the form
\[
J\bigl(K\cup B_\ve(x)\bigr) - J(K) = c(\ve) G_K(x) + o\bigl(c(\ve)\bigr)
\]
for some functions $c\colon \R_{>0} \to \R_{>0}$ and $G_K\colon \Omega\setminus K \to \R$.
Those $x \in \Omega\setminus K$ for which $G_K$ attains the largest negative values
are then added to the set $K$.
This process is iterated, until the function $G_K$ becomes non-negative everywhere.

In the case of the Mumford--Shah functional,
this procedure cannot applied directly,
because the functional is infinite whenever $K$ contains
a set of positive Lebesgue measure.
We therefore propose to use a different, though related,
functional for the computation of the asymptotic expansion,
which is based on an approximation of the one-dimensional Hausdorff measure:
The number of balls of radius $\ve > 0$ that are required
to cover a set $K$, multiplied with $2\ve$, provides a good estimate
of $\mathcal{H}^1(K)$ for $\ve$ small enough and $K$ sufficiently regular.
In the following, we introduce this approximating functional $J_{\ve,\kappa}$.
However, because we later prove the $\Gamma$-convergence of this functional
to $F$, it is necessary to let $J_{\ve,\kappa}$ depend on two functions,
the function $u$ and a piecewise constant edge indicator function $v$:

For each finite set $Y \subset \R^2$ and each $0 < \kappa < 1$
we define the function $v_{Y,\kappa}\colon \Omega \to \R$ by
\[
v_{Y,\kappa}(x) := \begin{cases}
\kappa & \text{ if } x \in \bigcup_{y \in Y} B_\ve(y)\,,\\
1 & \text{ else.}
\end{cases}
\]
For every $v \in L^2(\Omega)$ we define
\[
m_{\ve,\kappa}(v) := \inf\bigl\{\mathcal{H}^0(Y) : Y \subset \R^2,\ v = v_{Y,\kappa}\bigr\}\;.
\]
Here we set $m_{\ve,\kappa}(v) := +\infty$, if $v\neq v_{Y,\kappa}$
for any $Y \subset \R^2$.
Note that, for a given function $v$, there might exist different sets
$Y \subset \R^2$ such that $v = v_{Y,\kappa}$.

Finally, we introduce the family of functionals $J_{\ve,\kappa}\colon L^2(\Omega)\times L^2(\Omega)\to[0,+\infty]$, defined by
\begin{equation}
\label{functional_Je}
J_{\ve,\kappa}(u,v):= \frac{1}{2}\int_{\Omega}(u-f)^2\,dx
+\frac{\alpha}{2}\int_{\Omega} v \lvert\nabla u\rvert^2 \,dx
+ 2 \beta\ve\, m_{\ve,\kappa}(v)
\end{equation}
if $u\in H^1(\Omega)$, and $J_\ve(u,v):=+\infty$ otherwise.

In Section~\ref{sec:topological}, we derive an approximation of
the functional $J_{\ve,\kappa}$ that allows us to compute an approximation
of the minimizer using the idea of topological asymptotic expansions.
We show that
\[
J_{\ve,\kappa}(u,v_{Y\cup\{y\},\kappa}) - J_{\ve,\kappa}(u,v_{Y,\kappa})
\approx -\ve^2\pi \alpha\frac{1-\kappa}{1+\kappa}\abs{\nabla u(y)}^2+2\ve\beta\;.
\]
(cf.~Theorem~\ref{thm:Gasymptotic}).

In Section~\ref{sec:gamma}, we show that the functional $J_{\ve,\kappa}$ is indeed
an approximation of the Mumford--Shah functional in the sense of $\Gamma$-convergence.
More precisely, if we choose $\kappa = \kappa(\ve)$ in such a way
that $\kappa(\ve) = o(\ve)$ as $\ve \to 0$, then
\[
F = \Gammalim_{\ve \to 0} J_{\ve,\kappa(\ve)}\;.
\]
In particular, this implies that the minimizers of $J_{\ve,\kappa(\ve)}$ converge to a minimizer of $F$ as $\ve$ tends to zero. The adopted proof is based on the methods used for proving the $\Gamma$-convergence of the Ambrosio--Tortorelli approximation of $F$ (see~\cite{AmbTor90,Bra98}).
Finally, in Section~\ref{sec:algorithm}, we propose an algorithm to minimize the functional $J_{\ve,\kappa}$. We compare numerical results obtained with this algorithm with results obtained by minimization of the Ambrosio and Tortorelli model~\cite{AmbTor90}.

The present paper therefore provides both an approximation of the
Mumford--Shah functional and a concrete numerical method for its
approximate minimization. We note that numerical methods based on topological
analysis have recently been applied to image procedding problems
like edge detection (see~\cite{Aur07,Aur09,Mus09}). The cited papers,
however, do not note the connection to the existing variational
methods, which, as this paper shows, is very close indeed.


\section{Topological Asymptotic Analysis}
\label{sec:topological}

In this section, we derive the topological asymptotic expansion of the functional
$J_{\ve,\kappa}$ defined in \eqref{functional_Je}.
This expansion will form the basis of our numerical approach.

We assume that the function $f\in H^1(\Omega)$ is given and define the functional $G\colon L^2(\Omega)\times L^2(\Omega)\to[0,+\infty]$ by
\begin{equation*}
G(\tilde{u},\tilde{v}):= \frac{1}{2}\int_{\Omega}(\tilde{u}-f)^2\,dx
+\frac{\alpha}{2}\int_{\Omega} \tilde{v} \abs{\nabla \tilde{u}}^2 \,dx
\end{equation*}
if $\tilde{u}\in H^1(\Omega)$ and $\norm{\tilde{v}}_\infty < \infty$,
and $G(\tilde{u},\tilde{v}):=+\infty$ otherwise.

Now assume that $K$ is an open subset of $\Omega$ and $0 < \kappa < 1$
satisfying $\alpha\kappa < 1$.
We define the function $v\colon\Omega\to\R$ by
\[
v(x) =
\begin{cases}
\kappa &\text{ if } x \in K,\\
1 &\text{ else.}
\end{cases}
\]

Using standard methods of variational calculus one can show that
the mapping $\tilde{u} \mapsto G(\tilde{u},v)$ attains a unique
minimizer in $H^1(\Omega)$, which we denote by $u$.

The main result of this section is stated in the following theorem:
\begin{theorem}
\label{thm:Gasymptotic}
Let $K\subset\Omega$. For $\yt\in\Omega\setminus\bar{K}$ and $\ve > 0$ define
the functions
\[
\vt(x) :=
\begin{cases}
\kappa &\text{if } x \in K\cup B_{\ve}(\yt),\\
1 &\text{ else,}
\end{cases}
\]
and
\[
\ut := \argmin_{\tilde{u}\in H^1(\Omega)} G(\tilde{u},\vt)\,.
\]
Then for all compact subsets $L \subset \Omega\setminus \bar{K}$ we have
\[
 \sup_{\yt\in L}
\Big\lvert G(u^{(\yt,\ve)},v^{(\yt,\ve)})- G(u,v)
+ \alpha \ve^2 \pi \dfrac{1-\kappa}{1+\kappa}\abs{\nabla
  u(\yt)}^2\Big\rvert =
O(\ve^{5/2})\;.
\]
\end{theorem}

\begin{remark}
  We note that a very similar result has been derived in~\cite{Ams06}.
  The setting there, however, is slightly different (it amounts more or
  less to the case $K = \emptyset$). In addition, our result not only
  provides the asymptotics of the difference $G(u^{(\yt,\ve)},v^{(\yt,\ve)})- G(u,v)$
  but also the asymptotic size of the error term.
\end{remark}

For the remaining part of this section we assume that the compact
set $L \subset \Omega\setminus \bar{K}$ is fixed.
Moreover we define
\[
\delta := \frac{1}{3} \dist(L,\partial\Omega\cup\bar{K})
= \frac{1}{3}\min\bigl\{\abs{x-y} : x \in L,\ y \in \partial\Omega \cup\bar{K}\bigr\}\;.
\]
In addition, we consider the set
\[
\hat{L} := L + \bar{B}_\delta(0)
:= \bigl\{x+y : x \in L,\ \abs{y} \le \delta\bigr\}\;.
\]

Before we give the proof of the Theorem~\ref{thm:Gasymptotic},
we need to introduce some auxiliary results.
First, we recall that the assumptions that $u$ and $u^{(\hat{y},\ve)}$
are minimizers of $G(\cdot,v)$ and $G(\cdot,v^{(\hat{y},\ve)})$,
respectively, imply that
\begin{equation}\label{eq:phi}
\begin{aligned}
\int_\Omega (u^{(\hat{y},\ve)}-f)\varphi
+ \alpha v^{(\hat{y},\ve)} \langle \nabla u^{(\hat{y},\ve)},\nabla \varphi\rangle \,dx & = 0\,,\\
\int_\Omega (u-f)\varphi + \alpha v \langle \nabla u,\nabla \varphi\rangle \,dx & = 0\,,
\end{aligned}
\end{equation}
for all $\varphi \in H^1(\Omega)$.

We first we need a regularity result for the function $u$.

\begin{lemma}\label{le:C1}
The function $u$ satisfies
\[u \in C^{1,\lambda}_{\textrm{loc}}(\Omega\setminus \bar{K})\]
for all $\lambda\in(0,1)$. Moreover, there exists a constant $C_1$ only depending on $L$, $K$ and $\lambda$
such that
\[
\norm{\nabla u}_{L^\infty(\hat{L})} \le C_1\norm{f}_{H^1(\Omega)}\;.
\]
\end{lemma}

\begin{proof}
Because by assumption $f \in H^1(\Omega)$
and the function $v$ is constant on $\Omega \setminus K$,
it follows from standard theorems on the regularity
of solutions of elliptic equations that
$u \in H^3_{\text{loc}}(\Omega\setminus \bar{K})$
(see~\cite[Thm.~8.10]{GilTru01}). Moreover,
there exists a constant $c_1$ only depending on $\hat{L}$
(and therefore on $L$ and $K$) such
that $\norm{u}_{H^3(\hat{L})} \le c_1\norm{f}_{H^1(\Omega)}$.
Then, the Sobolev embedding theorem~\cite[Thm.~5.4]{Ada75}
implies that $u \in C^{1,\lambda}_{\text{loc}}(\Omega\setminus \bar{K})$
for all $\lambda\in(0,1)$, and consequently,
$$\norm{\nabla u}_{L^\infty(\hat{L})} \le \norm{u}_{L^\infty(\hat{L})}+\norm{\nabla u}_{L^\infty(\hat{L})}\leq c_2 \norm{u}_{H^3(\hat{L})} \le c_1 c_2\norm{f}_{H^1(\Omega)}$$
with the constant $c_2$ depending on $\hat{L}$ and $\lambda$.
\end{proof}

As a second step, we need $H^1$-norm and $L^2$-norm estimates of the difference $\ut-\uo$.
First we show that the $H^1$-norm of the difference $\ut-\uo$ on the whole
domain $\Omega$ is of order $\ve$.

\begin{lemma}
\label{H_1_norm}
There exists a constant $C_2 > 0$ only depending on $L$, $\kappa$,
$\Omega$, and $K$ such that for all $\ve > 0$ with $\ve < \delta$
and $\hat{y} \in L$ the estimate
\[
\norm{\ut-\uo}_{H^1(\Omega)}\leq C_2\norm{f}_{H^1(\Omega)}\ve
\]
holds.
\end{lemma}

\begin{proof}
Computing the difference between the two equations in~\eqref{eq:phi}
and using the definition of $v^{(\hat{y},\ve)}$, we obtain that
\begin{multline*}
\int_{\Omega}(\ut-\uo)\varphi\,dx+\alpha\int_{\Omega}\vt \langle\nabla(\ut-\uo),\nabla \varphi\rangle\,dx\\
= \alpha(1-\kappa)\int_{\Bt}\langle\nabla\uo,\nabla\varphi\rangle\,dx
\end{multline*}
for all $\varphi \in H^1(\Omega)$.
In particular, it follows with $\varphi = u^{(\hat{y},\ve)}-u$ that
\begin{multline*}
\alpha\kappa\norm{u^{(\hat{y},\ve)}-u}_{H^1(\Omega)}^2
\le \alpha(1-\kappa)\int_{\Bt}\langle\nabla\uo,\nabla(u^{(\hat{y},\ve)}-u)\rangle\,dx\\
\le \alpha(1-\kappa)\norm{\nabla u}_{L^2(B_\ve(\hat{y}))} \norm{u^{(\hat{y},\ve)}-u}_{H^1(\Omega)}\;.
\end{multline*}
Moreover Lemma~\ref{le:C1} implies that
\begin{equation}
\label{nablauL2norm}
\norm{\nabla u}_{L^2(B_\ve(\hat{y}))}
\le \norm{\nabla u}_{L^\infty(B_\ve(\hat{y}))} \sqrt{\pi} \ve
\le \norm{\nabla u}_{L^\infty(\hat{L})} \sqrt{\pi}\ve
\le C_1\sqrt{\pi}\norm{f}_{H^1(\Omega)} \ve\;.
\end{equation}
Setting $C_2 := (1-\kappa)\sqrt{\pi}C_1/\kappa$,
the assertion follows.
\end{proof}

\begin{lemma}
\label{H_1_norm_K}
There exists a constant $C_3$ only depending on $L$, $\kappa$, $\Omega$, $K$, and $\norm{f}_{H^1(\Omega)}$
such that
\[
\norm{\ut-\uo}_{L^2(\Omega)} \le C_3 \ve^{3/2}
\]
for every $\hat{y} \in L$ and $0 < \ve < \delta$.
\end{lemma}
\begin{proof}
Let $g\in H^1(\Omega)$ satisfy $\norm{g}_{H^1(\Omega)}\leq 1$ and assume that
$w^{(\hat{y},\ve)}$ and $w$ are the unique solutions to the equations (\ref{eq:phi})
but with given $g$ instead of $f$, that is
\begin{equation}\label{eq:ww}
\begin{aligned}
\int_\Omega (w^{(\hat{y},\ve)}-g)\varphi
+ \alpha v^{(\hat{y},\ve)} \langle \nabla w^{(\hat{y},\ve)},\nabla \varphi\rangle \,dx & = 0\,,\\
\int_\Omega (w-g)\varphi + \alpha v \langle \nabla w,\nabla \varphi\rangle \,dx & = 0\,,
\end{aligned}
\end{equation}
for all $\varphi \in H^1(\Omega)$.

Taking $\varphi = w^{(\hat{y},\ve)}$ and $\varphi =w$ in the first and the second equation in (\ref{eq:phi}), respectively, and next subtracting these equations from the corresponding equations in (\ref{eq:ww}) with $\varphi = u^{(\hat{y},\ve)}$ and $\varphi = u$, we obtain
\begin{equation*}
\begin{aligned}
\int_\Omega w^{(\hat{y},\ve)}\,f\,dx-\int_\Omega u^{(\hat{y},\ve)}\,g\,dx & = 0\,,\\
\int_\Omega w\, f\,dx-\int_\Omega u\,g\,dx & = 0\,.
\end{aligned}
\end{equation*}
In particular,
\begin{equation*}
\int_\Omega (u^{(\hat{y},\ve)}-u)\,g\,dx=\int_\Omega (w^{(\hat{y},\ve)}-w)\,f\,dx\,.
\end{equation*}
Next, we note that
\begin{multline}\label{aaa}
\int_{\Omega} (u^{(\hat{y},\ve)}-u)\,g\,dx
=\dfrac{1}{2}\left(\int_\Omega(u^{(\hat{y},\ve)}-u)\,g\,dx
+\int_\Omega (w^{(\hat{y},\ve)}-w)\,f\,dx\right)\\
=\dfrac{1}{2}\int_\Omega \left((u^{(\hat{y},\ve)}+w^{(\hat{y},\ve)})-(u+w)\right)\,(f+g)\,dx \\
-\dfrac{1}{2}\int_\Omega(u^{(\hat{y},\ve)}-u)\,f\,dx
-\dfrac{1}{2}\int_\Omega (w^{(\hat{y},\ve)}-w)\,g\,dx\,.
\end{multline}

Computing the difference of the two equations in~\eqref{eq:phi}
with $\varphi = u$ and $\varphi = u^{(\hat{y},\ve)}$, respectively,
we obtain
\begin{multline*}
\int_\Omega (\ut-\uo)\,f\,dx=\alpha\left(\int_{\Omega}\vo\langle\nabla\uo,\nabla \ut\rangle\,dx-\int_{\Omega}\vt\langle\nabla\uo,\nabla \ut\rangle\,dx\right)\\
\begin{aligned}
&=\alpha(1-\kappa)\int_{\Bt}\langle\nabla\uo,\nabla \ut\rangle \,dx\\
&=\alpha(1-\kappa)\left(\int_{\Bt}\langle\nabla\uo,\nabla (\ut-\uo)\rangle \,dx+\int_{\Bt}\abs{\nabla\uo}^2 \,dx \right)\,.
\end{aligned}
\end{multline*}
Application of the Cauchy-Schwarz inequality  to the above formula, and next, the estimate (\ref{nablauL2norm}) and Lemma \ref{H_1_norm} yields the inequality
\begin{equation}\label{ineq:GG}
\begin{aligned}
\Big\lvert \int_\Omega &(\ut-\uo)\,f\,dx \Big\rvert \\
&\leq \alpha(1-\kappa)\left(\norm{\nabla u}_{L^2(B_\ve(\hat{y}))} \norm{u^{(\hat{y},\ve)}-u}_{H^1(\Omega)}+ \norm{\nabla u}_{L^2(B_\ve(\hat{y}))}^2\right)\\
&\leq\alpha(1-\kappa)
\left(\sqrt{\pi} C_1 C_2+\pi C_1^2\right)\norm{f}_{H^1(\Omega)}^2\ve^2\,.
\end{aligned}
\end{equation}
In a similar manner, using the assumption that $\norm{g}_{H^1(\Omega)} = 1$,
we can show that
\begin{equation}
\label{ineq:GG1}
\begin{aligned}
\Big\lvert \int_\Omega &(w^{(\hat{y},\ve)}-w)g\,dx \Big\rvert
\leq \alpha(1-\kappa)
\left(\sqrt{\pi} C_1 C_2+\pi C_1^2\right)\ve^2\,,
\end{aligned}
\end{equation}
and
\begin{multline}
\label{ineq:GG2}
\Big\lvert \int_\Omega \bigl((u^{(\hat{y},\ve)}+w^{(\hat{y},\ve)})-(u+w)\bigr)(f+g)\,dx \Big\rvert\\
\leq \alpha(1-\kappa)\left(\sqrt{\pi}C_1 C_2+\pi C_1^2\right)
2(\norm{f}_{H^1(\Omega)}^2+1)\ve^2\,.
\end{multline}
Finally, combining \eqref{ineq:GG}, \eqref{ineq:GG1} and \eqref{ineq:GG2} with \eqref{aaa}, we obtain
\begin{equation*}
\begin{split}
\Big\lvert\int_{\Omega} (u^{(\hat{y},\ve)}&-u)g\,dx\Big\rvert\\
&\leq \frac{3\alpha(1-\kappa)}{2}
\left(\sqrt{\pi} C_1 C_2+\pi C_1^2\right)(\norm{f}_{H^1(\Omega)}^2+1)\ve^2\,.
\end{split}
\end{equation*}

Therefore, we have
\begin{multline}\label{eq:Hmone_estimate}
\norm{u^{(\hat{y},\ve)}-u}_{H^{-1}(\Omega)}
:= \sup\Bigl\{\Bigl\lvert\int_{\Omega} (u^{(\hat{y},\ve)}-u)g \,dx\Bigr\rvert :
g \in H^1(\Omega),\ \norm{g}_{H^1(\Omega)} \le 1\Bigr\}\\
\le c(\norm{f}_{H^1(\Omega)}^2 + 1)\ve^2
\end{multline}
with $c = 3 \alpha(1-\kappa)\left(\sqrt{\pi} C_1 C_2+\pi C_1^2\right)/2$.
Now, estimates from the theory of Hilbert scales
(see~\cite[Thm.~9.4]{KrePet66}), Lemma~\ref{H_1_norm}, and~\eqref{eq:Hmone_estimate}
imply that
\begin{equation*}
\begin{split}
\norm{u^{(\hat{y},\ve)}-u}_{L^2(\Omega)}^2
&\le \norm{u^{(\hat{y},\ve)}-u}_{H^{-1}(\Omega)}\norm{u^{(\hat{y},\ve)}-u}_{H^1(\Omega)}\\
&\le c(\norm{f}_{H^1(\Omega)}^2 + 1)\ve^2\, C_2\norm{f}_{H^1(\Omega)}\ve\;.
\end{split}
\end{equation*}
Therefore, the desired estimate holds with
$C_3=\sqrt{c(\norm{f}_{H^1(\Omega)}^2 + 1)C_2\norm{f}_{H^1(\Omega)}}$.
\end{proof}

In the next step, we need to derive an estimate for the function $\ut$ on the boundary $\partial\Bt$.
To do this, we follow Vogelius et al.~\cite{VogVol00}, where such estimate has been derived for the solution of the homogeneous Helmholtz equation with Dirichlet boundary conditions.

We first introduce the Green function corresponding to the equation $u-\alpha\Delta u=f$ on $\Omega\setminus\bar{K}$
with Neumann boundary conditions. That is, the function $N(\cdot,y)$ solves the problem
\begin{equation*}
  \begin{cases}
    N(x,y)-\alpha \Delta_x N(x,y)=\delta_y(x) & x\in\Omega\setminus\bar{K}\,,\\[0.2cm]
    \vp{N}{n}(x,y)=0 & x\in\partial(\Omega\setminus \bar{K})\,,\\
  \end{cases}
\end{equation*}
for $y\in\Omega$. We note that $N(\cdot,y)$ can be written as the sum of the fundamental solution
corresponding to the equation $u-\alpha\Delta u=\delta_y$, denoted by $\Gamma(\cdot,y)$,
and a corrector function $h(\cdot,y)$, which is chosen in such a way that the normal derivative
of $N(\cdot,y)$ vanishes on the boundary $\partial(\Omega\setminus \bar{K})$.
The function $\Gamma(\cdot,y)$ is given by
\[
\Gamma(x,y)=\frac{1}{2\pi}K_0\left(\dfrac{1}{\sqrt{\alpha}}|x-y|\right)
\]
for all $x$, $y\in\mathbb{R}^2$, such that $x\neq y$.
Here $K_0$ denotes the modified Bessel function of the second kind (see, e.g., \cite[p.~490]{Pol02}).
Furthermore, the function $K_0$ has an asymptotic expansion of the form
\begin{equation*}
K_0(z)=-\ln z + \ln 2-\gamma +O(z^2\abs{\ln z})
\end{equation*}
for $z\to 0$, where $\gamma$ denotes the Euler-Mascheroni constant  (see, e.g., \cite[Ch.~51]{Spa87}).
Therefore, we conclude that $\Gamma(\cdot,y)$ can be approximated as
\begin{equation}
\label{fundamental_approx}
\Gamma(x,y)=\frac{1}{2\pi}\Bigl(-\ln(|x-y|)+\dfrac{1}{2}\ln\alpha+\ln 2-\gamma\Bigr)+O(\abs{x-y}^2\abs{\ln\abs{x-y}})\,,
\end{equation}
when $\abs{x-y}\to 0$. Moreover, we observe that $\Gamma(\cdot,y)$ has the same singular behavior as
the fundamental solution of the Laplace equation
\begin{equation*}
\Phi(x,y)=-\frac{1}{2\pi}\ln(|x-y|)
\end{equation*}
defined for all $x$, $y\in\mathbb{R}^2$, such that $x\neq y$.
We need the approximation \eqref{fundamental_approx} in order to be able to apply standard methods of potential theory
to derive an estimation for the function $\ut$ on $\partial\Bt$. Such way of proceeding is
common when dealing with problems of this kind (see, e.g., Colton and Kress \cite{Col83,Col98}).

\begin{lemma}
\label{lemma1}
There exists a constant $C_4$ only depending on $L$, $K$, and $\Omega$, such that for every point
$\hat{y}\in L$, $0<\ve<\delta$, and $y\in\Omega$ satisfying $\ve<\abs{y-\hat{y}}<2\ve$ the estimate
\[
\norm{N(\cdot,y)}_{L^2(\veps{B})}\leq C_4 \ve\abs{\ln\ve}
\]
holds.
\end{lemma}
\begin{proof}
Since by elliptic regularity $h\in C^{\infty}(\Omega\setminus K,\Omega\setminus\bar{K})$
and
\[
\dist(\hat{y},\Omega\setminus K) \ge 3\delta > 3\ve\,,
\]
we have
\[
\norm{h(\cdot,y)}_{L^2(\Bt)}\leq \sqrt{\pi} \norm{h(\cdot,y)}_{L^{\infty}(\Omega\setminus\bar{K})} \ve \leq \sqrt{\pi} c_1 \ve
\]
with $c_1 > 0$ only depending on $L$, $K$, and $\Omega$.
The estimate $\norm{\Phi(\cdot,y)}_{L^2(\Bt)}\leq c_2 \ve\abs{\ln\ve}$
for some $c_2$ only depending on $\ve$ and
for all $y$ satisfying $\ve < \abs{y-\hat{y}} < 2\ve$
can be easily derived in the polar coordinate system.
Therefore, from the Minkowski inequality and \eqref{fundamental_approx} we get that
\begin{multline*}
\norm{N(\cdot,y)}_{L^2(\Bt)} \le
\norm{h(\cdot,y)}_{L^2(\Bt)}+\norm{\Gamma(\cdot,y)}_{L^2(\Bt)}\\
\leq \sqrt{\pi} c_1 \ve + c_2 \ve \abs{\ln\ve}
\leq C_4 \ve \abs{\ln\ve}
\end{multline*}
with $C_4$ chosen slightly larger than $c_2$.
\end{proof}

\begin{lemma}
\label{lemma2}
There exists a constant $C_5$ only depending on $L$, $K$, $\Omega$ $\kappa$, $\alpha$, and $f$,
such that for every point $\hat{y}\in L$, $0<\ve<\delta$, and $y\in\Omega$ satisfying $\ve<\abs{y-\hat{y}}<2\ve$ the estimate
\begin{equation}
\label{boundary_integral_formulation}
\Big\lvert \ut(y)-\uo(y)-\alpha(1-\kappa)\int_{\partial\Bt}\ut(x) \vp{N}{n}(x,y)\,ds(x) \Big\rvert \leq C_5 \ve^{3/2}
\end{equation}
holds.
\end{lemma}
\begin{proof}
Using standard calculations (see, e.g., \cite[p.~33]{Evans10}) and
approximation~(\ref{fundamental_approx}), we can derive the integral
representation formulas for the functions $u$ and $\ut$
\begin{equation}
\label{int_repres_u1}
\begin{split}
&\uo(y)=
-\alpha\int_{\partial \Bo} N(x,y)
\vp{\uo}{n}^{+}(x)\,ds(x)
+\int_{\Omega\setminus\bar{K}} N(x,y)\,f(x)\,dx
\end{split}
\nonumber
\end{equation}
for $y\in\Omega\setminus \bar{K}$, and
\begin{equation}
\label{int_repres_u2}
\begin{split}
&\ut(y)=\alpha\int_{\partial \Bt} \left(\ut(x)\vp{N}{n}(x,y)-N(x,y)
\vp{{\ut}^{+}}{n}(x)\right)\,ds(x)\\
&\quad-\alpha\int_{\partial \Bo} N(x,y)
\vp{{\ut}^{+}}{n}(x)\,ds(x)+\int_{\omegt} N(x,y)\,f(x)\,dx
\end{split}
\nonumber
\end{equation}
for $y\in\omegt$. Computing the difference of two above formulas we obtain
\begin{equation}
\label{u2minusu1}
\begin{split}
\ut(y)&-\uo(y)=\\
&\alpha\int_{\partial \Bt} \left(\ut(x)\vp{N}{n}(x,y)-N(x,y)
\vp{{\ut}^{+}}{n}(x)\right)\,ds(x)\\
&-\alpha\int_{\partial \Bo} N(x,y)
\vp{}{n}\bigl({\ut}^{+}(x)-{\uo}^{+}(x)\bigr)\,ds(x)\\
&-\int_{\Bt} N(x,y)\,f(x)\,dx
\end{split}
\end{equation}
for $y\in\omegt$.

Using the transmission condition that $\ut$ satisfies on
$\partial\Bt$, the Green formula, and that
\begin{equation}
\label{greeneq}
N(x,y)-\alpha\Delta N(x,y)=0
\end{equation}
for $x\in\Bt$ and $y\in\omegt$, we have
\begin{equation}
\begin{split}
\label{estimB}
&\int_{\partial\Bt} N(x,y)\vp{{\ut}^{+}}{n}(x)\,ds(x)
=\kappa \int_{\partial\Bt} N(x,y)\vp{{\ut}^{-}}{n}(x)\,ds(x)\\
&=\kappa\int_{\Bt}N(x,y)\,\Delta\ut(x)\,dx+\kappa\int_{\Bt}\bigl\langle\nabla N(x,y),\nabla\ut(x)\bigr\rangle\,dx\\
&=\dfrac{1-\kappa}{\alpha}\int_{\Bt}N(x,y)\ut(x)\,dx
-\dfrac{1}{\alpha}\int_{\Bt}N(x,y)f(x)\,dx \\
&\quad+\kappa\int_{\partial\Bt}\ut(x) \vp{N}{n}(x,y)\,ds(x)\,.
\end{split}
\end{equation}
Using similar arguments as above and applying the Cauchy-Schwarz inequality, we obtain
\begin{multline}\label{estimK1}
-\int_{\partial \Bo}N(x,y)
\vp{}{n}\bigl({\ut}^{+}(x)-{\uo}^{+}(x)\bigr)\,ds(x)\\
\begin{aligned}
&=-\kappa\int_{\partial \Bo} N(x,y)
\vp{}{n}\bigl({\ut}^{-}(x)-{\uo}^{-}(x)\bigr)\,ds(x)\\
&=(1-\kappa)\int_{K}\bigl(\ut(x)-\uo(x)\bigr)N(x,y)\,dx\\
&\leq (1-\kappa)\norm{\ut-\uo}_{L^2(K)} \norm{N(\cdot,y)}_{L^2(K)}\,.
\end{aligned}
\end{multline}
Because $\dist(y,K) \ge \dist(\hat{y},K) - 2\ve > \delta$,
there exists some $c > 0$ such that $\norm{N(\cdot,y)}_{L^2(K)}<c$.
Application of Lemma \ref{H_1_norm_K} therefore yields
\begin{equation}
\label{estimK}
\norm{\ut-\uo}_{L^2(K)}\norm{N(\cdot,y)}_{L^2(K)}
\leq c \norm{\ut-\uo}_{L^2(\Omega)}
\leq C_3 c \, \ve^{3/2}\,.
\end{equation}
Taking into account \eqref{estimB}, \eqref{estimK1}, and \eqref{estimK} in \eqref{u2minusu1}, we obtain
\begin{equation*}
\begin{aligned}
\ut(y)-\uo(y)=&\alpha(1-\kappa)\int_{\partial\Bt}\ut(x) \vp{N}{n}(x,y)\,ds(x)\\
&-(1-\kappa)\int_{\Bt}\ut(x)N(x,y)\,dx\\
&+(1-\alpha)\int_{\Bt}f(x)N(x,y)\,dx + O(\ve^{3/2})\,,
\end{aligned}
\end{equation*}
the last term being bounded by $\alpha(1-\kappa)C_3 c\,\ve^{3/2}$.
To estimate the remaining integrals, we note that
\begin{equation}
\label{ueL2norm}
\norm{\ut}_{L^2(\Bt)}\leq
\sqrt{\pi} \norm{f}_{L^{\infty}(\Omega)}\,\ve\,.
\end{equation}
Next, using the Cauchy--Schwarz inequality and Lemma \ref{lemma1} we get
\begin{equation*}
\begin{split}
\int_{\Bt} \ut(x) N(x,y)\,dx
&\leq\norm{\ut}_{L^2(\Bt)}\norm{N(\cdot,y)}_{L^2(\Bt)}\\
&\leq \sqrt{\pi} \norm{f}_{L^{\infty}(\Omega)}C_4\, \ve^2\abs{\ln\ve}\,.
\end{split}
\end{equation*}
In a similar way, we show that
\begin{equation*}
\begin{split}
\int_{\Bt} f(x) N(x,y)\,dx
&\leq\norm{f}_{L^2(\Bt)}\norm{N(\cdot,y)}_{L^2(\Bt)}\\
&\leq \sqrt{\pi} \norm{f}_{L^{\infty}(\Omega)}C_4 \,\ve^2\abs{\ln\ve}\,.
\end{split}
\end{equation*}
Therefore, we obtain the desired estimate with
$C_5$ slightly larger than $\alpha(1-\kappa)C_3 c$.
\end{proof}

\begin{lemma}
\label{lemma3}
There exists a constant $C_6$ only depending on $L$, $K$, $\kappa$, $\alpha$, and $f$,
such that for every point $\hat{y}\in L$, $0<\ve<\delta$, and $y\in\Omega$ satisfying $\ve<\abs{y-\hat{y}}<2\ve$ the estimate
\[
\Big\lvert\int_{\partial\Bt}\ut(x) \vp{N}{n}(x,y)\,ds(x)-\int_{\partial \Bt} \ut(x)\vp{\Phi}{n}(x,y)\,ds(x)\Big\rvert \leq C_6 \ve^2\abs{\ln\ve}
\]
holds.
\end{lemma}

\begin{proof}
Denote $H(x,y):=h(x,y)+\Gamma(x,y)-\Phi(x,y)$. Then
$N(x,y)=\Phi(x,y)+H(x,y)$ and, using \eqref{greeneq}, we obtain
\begin{equation}
\label{greeneqh}
N(x,y)=\alpha\Delta H(x,y)\,.
\end{equation}
Application of the Green formula, \eqref{greeneqh}, and the Cauchy-Schwarz inequality
yield
\begin{multline*}
\Bigl\lvert\int_{\partial\Bt}\ut(x)\vp{H}{n}(x,y)\,ds(x)\Bigr\rvert\\
\begin{aligned}
&=\Bigl\lvert\int_{\Bt}\ut(x)\Delta H(x,y)\,dx
+\int_{\Bt}\bigl\langle \nabla \ut(x), \nabla H(x,y)\bigr\rangle\,dx\Bigr\rvert\\
&=\Bigl\lvert\dfrac{1}{\alpha}\int_{\Bt}\ut(x)N(x,y)\,dx
+\int_{\Bt}\bigl\langle \nabla \ut(x), \nabla H(x,y)\bigr\rangle\,dx\Bigr\rvert\\
&\leq
\dfrac{1}{\alpha}
\norm{\ut}_{L^2(\Bt)}\norm{N(\cdot,y)}_{L^2(\Bt)}\\
&\quad+\norm{\nabla (\ut-\uo)}_{L^2(\Bt)}\,\norm{\nabla H(\cdot,y)}_{L^2(\Bt)}\\
&\quad+\norm{\nabla\uo}_{L^2(\Bt)}\,\norm{\nabla H(\cdot,y)}_{L^2(\Bt)}\;.
\end{aligned}
\end{multline*}
Using \eqref{ueL2norm} and Lemmas \ref{H_1_norm} and~\ref{lemma1}, it follows that
\begin{multline}\label{lemma3:1}
\Bigl\lvert\int_{\partial\Bt}\ut(x)\vp{H}{n}(x,y)\,ds(x)\Bigr\rvert\\
\le C_4\dfrac{\sqrt{\pi}}{\alpha}\norm{f}_{L^\infty(\Omega)}\ve^2\abs{\ln\ve}
+ (C_2+C_1\sqrt{\pi})\norm{f}_{H^1(\Omega)}\norm{\nabla H(\cdot,y)}_{L^2(\Bt)}\ve\;.
\end{multline}
Using the explicit form of $\Gamma(\cdot,y)$ and $\Phi(x,y)$, and next
the Taylor expansion of the function $K_0$ (see \cite[Ch.~51]{Spa87}), we can show that
\[
\abs{\nabla \Gamma(x,y) - \nabla \Phi (x,y)} \leq c\, \ve\abs{\ln\ve}
\]
for some $c > 0$ only depending on $\delta$
and all $x \in B_{\ve}(\hat{y})$ and $y$ satisfying $\ve < \abs{y-\hat{y}} < 2\ve$.
Thus
\[
\norm{\nabla H(\cdot,y)}_{L^2(\Bt)} \le (c\ve\abs{\ln\ve}+\norm{h}_{L^\infty(\Bt)}) \sqrt{\pi}\ve\;.
\]
Thus we obtain the required estimate from~\eqref{lemma3:1}
by choosing $C_6$ slightly larger than $C_4\sqrt{\pi}\norm{f}_{L^\infty(\Omega)}/\alpha$.
\end{proof}

From Lemma \ref{lemma2} and Lemma \ref{lemma3} and using the jump formula for the double layer potential (see e.g. Kress \cite[p. 68]{Kre89}), we have
\begin{equation}
\label{ujump}
\begin{split}
\Big\lvert\dfrac{1}{2}\alpha(1+\kappa)&\ut(y)-\uo(y)\\
-&\alpha(1-\kappa)\int_{\partial\Bt}\ut(x) \vp{\Phi}{n}(x,y)\,ds(x)\Big\rvert
\leq C_7\,\ve^{3/2}
\end{split}
\end{equation}
for $y\in\partial\Bt$, with $C_7 > C_5$.

Now we introduce the auxiliary vector valued function $\phi\colon\R^2\to\R\times\R$ that solves the problem
\begin{equation}
\label{vector_function}
\left\{\begin{array}{ll}
\Delta \phi (x)=0 & x\in\mathbb{R}^2\setminus \bar{B}(0) \ \mbox{or} \  x\in B(0) \,, \\[0.2cm]
\phi^{+}(x)=\phi^{-}(x) & x \in \partial B(0) \,, \\[0.2cm]
\vp{\phi^{+}}{n}(x)- \kappa\vp{\phi^{-}}{n}(x)=-\kappa n(x) & x \in \partial B(0) \,,\\[0.3cm]
\displaystyle{\lim_{\abs{x}\rightarrow \infty} \phi(x)=0} \,. \\
\end{array} \right.
\end{equation}
Here $B(0)$ denotes the unit ball in $\R^2$ of a center in $0$ and $n$ is the unit normal vector exterior to the boundary $\partial B(0)$.
The existence and uniqueness of a solution to the problem \eqref{vector_function}
can be proved using single layer potentials with suitably chosen densities.
For details, see Ammari and Kang~\cite{Amm04} or Cedio-Fengya et al.~\cite{CedMosVog98}.

Solving the problem \eqref{vector_function} using standard methods of potential theory, we obtain the explicit form of $\phi$, which reads
\begin{equation}
\label{explicit_phi}
\phi(x)=\dfrac{\kappa}{\kappa
+1} x \quad \text{and} \quad \phi(x)=\dfrac{\kappa}{\kappa
+1} \dfrac{x}{\abs{x}^2}
\end{equation}
for all $x\in B(0)$ and $x\in\R^2\setminus\bar{B}(0)$, respectively.

The result on asymptotic expansion of the function $\ut$ on the boundary $\partial\Bt$ is stated in the following Lemma:

\begin{lemma}
\label{volkov_vogelius}
For every point $\hat{y}\in L$, $0<\ve<\delta$, and $y\in\partial\Bt$ the estimate
\[
\Bigg\lvert \ut(y)-\uo(y) -\ve \left(\dfrac{1}{\kappa}-1\right)\langle\phi(y/\ve),\nabla \uo(\yt)\rangle \Bigg\rvert \leq C_7\ve^{3/2}
\]
holds, where the constant $C_7$ is as in~\eqref{ujump}.
\end{lemma}
\begin{proof}
The proof of this lemma can be proved starting from the formula~\eqref{ujump}
in the same way as in Vogelius and Volkov~\cite[Prop.~3]{VogVol00}.\end{proof}

Using all the above results, we can now prove Theorem \ref{thm:Gasymptotic}.

\subsection{Proof of Theorem \ref{thm:Gasymptotic}}
Using~\eqref{eq:phi} with $\varphi = u^{(\hat{y},\ve)}$ and $\varphi = u$,
we obtain that
\begin{equation}\label{Gasymp0}
G(\ut,\vt) - G(\uo,\vo)
= -\frac{1}{2}\int_\Omega f(\ut-\uo)\,dx\;.
\end{equation}
Again using~\eqref{eq:phi}, it follows that
\begin{multline}\label{Gasymp1}
-\frac{1}{2}\int_\Omega f(\ut-\uo)\,dx\\
\begin{aligned}
&=\frac{1}{2}\int_{\Omega}\vt\langle\nabla\uo,\nabla \ut\rangle\,dx
-\frac{1}{2}\int_{\Omega}\vo\langle\nabla\uo,\nabla \ut\rangle\,dx\\
&=-\frac{1}{2}\alpha(1-\kappa)\int_{\Bt}\langle\nabla\uo,\nabla \ut\rangle \,dx\;.
\end{aligned}
\end{multline}
Next, we apply Green's formula to $\uo\in H^1(\Omega)$
and use the fact that $u$ solves the equation $u-f = \alpha\varDelta u$
on $B_\ve(\hat{y})$ to obtain that
\begin{equation*}
\begin{split}
\alpha\int_{\Bt}&\langle\nabla\uo,\nabla \ut\rangle \,dx \\
&= \alpha\int_{\Bt} \langle\nabla(\ut-\uo),\nabla \uo\rangle\,dx
+\alpha\int_{\Bt} \abs{\nabla \uo}^2\,dx\\
&= \int_{\Bt}(\ut-\uo)(f-\uo)\,dx
+ \alpha\int_{\partial \Bt} (\ut-\uo)\vp{\uo}{n}\,ds\\
&\quad+ \alpha\int_{\Bt} \abs{\nabla \uo}^2\,dx\,.
\end{split}
\end{equation*}
To estimate the first integral on the right hand side of above equation,
we use the Cauchy-Schwarz inequality and Lemma~\ref{H_1_norm_K} and obtain
\begin{equation}
\label{estim1}
\begin{aligned}
\Big\lvert\int_{\Bt}(\ut-\uo)(\uo-f)\,dx\Big\rvert
&\le\norm{\uo-f}_{L^2(\Bt)}\norm{\ut-\uo}_{L^2(\Bt)}\\
&\le 2\norm{f}_{L^\infty(\Omega)} \sqrt{\pi} C_3\ve^{5/2}\;.
\end{aligned}
\end{equation}
Lemma~\ref{le:C1} implies that $\nabla u$ is H\"older continious on
every compact subset of $\Omega\setminus\bar{K}$ with the exponent
$\lambda = 1/2$.
Thus there exists some constant $c$ such that
\begin{equation}\label{contu}
\sup_{x \in B_\ve(\hat{y})} \frac{\abs{\nabla u(x)-\nabla u(\hat{y})}}{\abs{x-\hat{y}}^{1/2}}
\le c
\end{equation}
for all $\hat{y} \in L$ and $0 < \ve < \dist(L,\Omega\setminus\bar{K})/2$.
In particular, we have the estimate
\begin{equation}\label{estim2}
\begin{aligned}
\Bigl|\int_{\Bt} \abs{\nabla \uo}^2\,dx &- \ve^2\pi \abs{\nabla u(\hat{y})}^2\Bigr|\\
&= \Bigl| \int_{B_\ve(\hat{y})} \langle \nabla u-\nabla u(\hat{y}),\nabla u + \nabla u(\hat{y})\rangle\,dx\Bigr|\\
&\le \,c\, \ve^{1/2} \int_{B_\ve(\hat{y})} \abs{\nabla u + \nabla u(\hat{y})}\,dx\\
&\le 2\,c\, \pi C_1 \norm{f}_{H^1(\Omega)}\ve^{5/2}\;.
\end{aligned}
\end{equation}
The change of variable $x = \hat{y} + \ve \tilde{x}$,
application of Lemma \ref{volkov_vogelius} and the property \eqref{contu} yield
\begin{multline}
\label{estim3}
\int_{\partial \Bt} \bigl(\ut(x)-u(x)\bigr)
\frac{\partial u}{\partial n}(x)\,ds(x)\\
\begin{aligned}
&=\ve\int_{\partial B(0)} \bigl(\ut(\hat{y}+\ve\tilde{x})-u(\hat{y}+\ve\tilde{x})\bigr)
\frac{\partial u}{\partial n}(\hat{y}+\ve\tilde{x})\,ds(\tilde{x})\\
&=\ve^2 \left(\dfrac{1}{\kappa}-1 \right)\nabla\uo(\yt)^T \int_{\partial B(0)} \phi(\tilde{x}) \frac{\partial u}{\partial n}(\hat{y}+\ve\tilde{x})\,ds(\tilde{x})+O(\ve^{5/2})\\
&=\ve^2 \left(\dfrac{1}{\kappa}-1 \right)\nabla\uo(\yt)^T \left(\int_{\partial B(0)} \phi(\tilde{x}) n(\tilde{x})^T ds(\tilde{x})\right) \nabla \uo(\yt)+O(\ve^{5/2})\,,
\end{aligned}
\end{multline}
where $T$ denotes the vector transpose. Combining the estimates (\ref{estim1}), (\ref{estim2}) and (\ref{estim3}) with \eqref{Gasymp0} and \eqref{Gasymp1}, we obtain
\begin{equation}
\label{Gasymp2}
\begin{split}
G&(\ut,\vt)-G(\uo,\vo)\\
&=-\ve^2\dfrac{1}{2}\alpha(1-\kappa)\left(\dfrac{1}{\kappa}-1 \right)\nabla\uo(\yt)^T \left(\int_{\partial B(0)} \phi(\tilde{x}) n(\tilde{x})^T ds(\tilde{x})\right) \nabla \uo(\yt)\\
&\quad-\ve^2\dfrac{1}{2}\alpha(1-\kappa)\, \pi \, \abs{\nabla \uo(\yt)}^2 +
O(\ve^{5/2}) \,.
\end{split}
\end{equation}

Inserting the explicit formula for the function $\phi$, given in (\ref{explicit_phi}), to the asymptotic expansion (\ref{Gasymp2}), we get
\begin{equation*}
G(\ut,\vt)-G(\uo,\vo)=-\ve^2 \pi \alpha\dfrac{1-\kappa}{1+\kappa} \, \abs{\nabla \uo(\yt)}^2 + O(\ve^{5/2}) \,,
\end{equation*}
which ends the proof of Theorem \ref{thm:Gasymptotic}.


\section{The $\Gamma$-convergence of  $J_{\ve,\kappa(\ve)}$ to $F$}
\label{sec:gamma}

In this section we give a detailed proof that the sequence of functionals
$J_{\ve,\kappa}$ converges to the Mumford--Shah functional
in the sense of $\Gamma$-convergence,
if the parameter $\kappa$, depending on $\ve$,
tends sufficiently fast to zero as $\ve \to 0$.

We note that the $\Gamma$-convergence of a similar family of
functionals has been shown in~\cite{BraChaSol07}. As in this paper,
the authors approximate the edge set of $u$ by some set $K_\ve$,
the measure of which tends to $0$ as $\ve \to 0$.
The length of the edge set, however, is approximated by half of
the length of the boundary of this set $K_\ve$.
Instead, we approximate the length of the edge set by counting
the number of balls covering $K_\ve$, which, apparently,
is easier to compute. On the other hand, the results
in~\cite{BraChaSol07} apply to more general settings and,
in particular, also hold in higher dimensions, where our
approach fails.
Also, in the one-dimensional setting, the functionals treated
in~\cite{BraChaSol07} and the functional $J_\ve$ of this paper
are almost the same, which allows us to use some of the results
from~\cite{BraChaSol07} in our proofs.

Let us first recall the definition of the $\Gamma$-limit:
\begin{definition}
  Let $\X$ be a topological space and
  $J_j\colon \X\to [0,+\infty]$ a sequence of functionals on $\X$.
  Denote moreover, for $x \in \X$, by $\mathcal{N}(x)$ the
  set of all open neighborhoods of $x$.
  Then the \emph{$\Gamma$-lower limit} and the \emph{$\Gamma$-upper limit}
  of $J_j$ are the functionals defined as
  \[
  \begin{aligned}
    (\Gammaliminf_j J_j)(u) &:= \sup_{\U \in \mathcal{N}(u)} \liminf_j \inf_{v \in \U} J_j(v)\,,\\
    (\Gammalimsup_j J_j)(u) &:= \sup_{\U \in \mathcal{N}(u)} \limsup_j \inf_{v \in \U} J_j(v)\;.
  \end{aligned}
  \]
  If the $\Gamma$-upper and lower limits coincide, we define
  the $\Gamma$-limit by
  \[
  (\Gammalim_j J_j)(u) = (\Gammalimsup_j J_j)(u) = (\Gammaliminf_j J_j)(u)\;.
  \]
\end{definition}

In metric spaces, the $\Gamma$-limit of a sequence of functionals
can be characterized by means of the following result:

\begin{lemma}\label{le:gammalimit}
  Let $\X$ be a metric space
  and $J_j\colon \X\to [0,+\infty]$ a sequence of functionals on $\X$.
  Let moreover $J\colon \X \to [0,+\infty]$
  and let $\hat{\X}$ be a dense subset of $\{u \in \X:J(u)<\infty\}$.
  Assume moreover that for every $u \in \X$ there exists a sequence
  $u_j \to u$ with $u_j \in \hat{\X}$ such that
  $J(u_j) \to J(u)$.
  Then $J = \Gammalim_j J_j$, if the following conditions hold:
  \begin{enumerate}
  \item For every $u \in \X$ and every sequence $u_j \to u$
    with $u_j \in \X$ we have
    \begin{equation}\label{eq:liminf}
    J(u) \le \liminf_j J_j(u_j)\;.
    \end{equation}
  \item For every $u \in \hat{\X}$ and every $\delta > 0$ there
    exists a sequence $u_j \to u$ with $u_j \in \hat{\X}$
    for all $j$ such that
    \begin{equation}\label{eq:limsup}
    J(u) \ge \limsup_j J_j(u_j) - \delta\;.
    \end{equation}
  \end{enumerate}
\end{lemma}

\begin{proof}
  The result follows by combining the standard metric characterization
  of the $\Gamma$-limit (see, e.g., \cite[Prop.~8.1]{Dal93})
  with a diagonal sequence argument.
\end{proof}

Define
\begin{equation}\label{functional_MS_weak}
F(u,v)
:= \frac{1}{2} \int_\Omega (u-f)^2\,dx
  + \frac{\alpha}{2} \int_{\Omega\setminus S_u} \abs{\nabla u}^2\,dx
  + \beta\mathcal{H}^1(S_u)\;.
\end{equation}

The main result of this section is the following theorem:

\begin{theorem}\label{theorem1}
  Let $F$ and $J_{\ve,\kappa}$ be as
  in~\eqref{functional_MS_weak} and~\eqref{functional_Je}, respectively.
  Assume moreover that $\kappa(\ve) = o(\ve)$ as $\ve \to 0$.
  Then we have for every sequence $\ve_j \to 0$ that
  \[
  F = \Gammalim_j J_{\ve_j,\kappa(\ve_j)}\;.
  \]
\end{theorem}

We now prove Theorem~\ref{theorem1}
using the methods introduced by Ambrosio and Tortorelli \cite{Amb90},
as presented in the notes by Chambolle \cite{Cha00} and the books by Braides \cite{Bra98,Bra02}.
The proof is split into three parts.
In the first part, we will prove the lower bound, inequality~\eqref{eq:liminf},
in the one-dimensional case.
In the second part, we will extend this result to dimension 2 using the \emph{slicing method}
(see, e.g., \cite{Bra98, Bra02}).
In the last part, we will prove the upper bound, inequality~\eqref{eq:limsup}.

\subsection{The Lower Bound for $n=1$}

Let the set $\Omega \subset \R$ be open and bounded, and $f \in L^\infty(\Omega)$.
We define the one-dimensional Mumford--Shah functional
$\tilde{F}\colon L^2(\Omega)\times L^2(\Omega) \to [0,+\infty]$ as
\[
\tilde{F}(u,v) = \frac{1}{2} \int_\Omega (u-f)^2\,dx + \frac{\alpha}{2}\int_{\Omega\setminus S_u} (u')^2\,dx + \beta\mathcal{H}^0(S_u)
\]
if $u'$ is square integrable outside the jump set $S_u$ of $u$ and $v \equiv 1$;
otherwise $\tilde{F}(u,v)=+\infty$.

Because of technicalities of the proof of the two-dimensional case
that result from the restriction of the approximating functionals $J_{\ve,\kappa}$ to lines,
it is necessary to use a slightly different definition in the one-dimensional case;
instead of only covering the edge set
with balls of radius $\ve$, we also allow covers with smaller balls.
For each finite set $Y=\{y_i : 1\le i \le m\}$ of points in $\R$
we denote by $M_{\ve,\kappa}(Y,\Omega)$ the set of all functions $v\colon\Omega\to\R$
for which there exists a sequence
$\{\delta_i\}_{i=1}^m$ of positive numbers smaller than, or equal to, $\ve$,
such that for all $x\in \Omega$ we have
\[
v(x) =
\begin{cases}
\kappa &\text{ if } x \in \bigcup_{i=1}^{m}B_{\delta_i}(y_i),\\
1 &\text{ else.}
\end{cases}
\]
Furthermore we denote by
\begin{equation}\label{eq:tildemve}
\tilde{m}_{\ve,\kappa}(v,\Omega) := \inf\bigl\{\mathcal{H}^0(Y) : Y \subset \R,\ v \in M_{\ve,\kappa}(Y,\Omega)\bigr\}\;.
\end{equation}
As in the two-dimensional case,
we define $\tilde{m}_{\ve,\kappa}(v,\Omega):=+\infty$ if $v \not\in M_{\ve,\kappa}(Y,\Omega)$ for any $Y \subset\R^2$.

Finally, we define the functional $\tilde{J}_{\ve,\kappa}\colon L^2(\Omega)\times L^2(\Omega) \to [0,+\infty]$
as
\[
\tilde{J}_{\ve,\kappa}(u,v) := \frac{1}{2}\int_\Omega (u-f)^2\,dx
+ \frac{\alpha}{2}\int_\Omega v(u')^2\,dx + \beta \tilde{m}_{\ve,\kappa}(v,\Omega)
\]
if $u \in H^1(\Omega)$ and $v \in M_{\ve,\kappa}(Y,\Omega)$, otherwise $\tilde{J}_{\ve,\kappa}(u,v) :=+\infty$.
For proving the inequality
\[
\tilde{F}(u,v) \le \liminf_{j\to\infty} \tilde{J}_{\ve_j,\kappa(\ve_j)}(u_j,v_j)\;.
\]
for all sequences $(u,v) \in L^2(\Omega)\times L^2(\Omega)$ converging
to $(u,v) \in L^2(\Omega)\times L^2(\Omega)$ and $\ve_j \to 0$,
we can basically rely on the results and techniques 
from~\cite[Proposition~3]{BraChaSol07}, where the same result has been
shown in an only slightly different setting.
We therefore omit the proof.

\subsection{The Lower Bound for $n = 2$}

The second part of the proof of Theorem~\ref{theorem1}
is concerned with showing~\eqref{eq:liminf} for $\Omega \subset \R^2$.
The proof applies the slicing method following Braides \cite{Bra98,Bra02}.
To that end it is necessary to introduce some notational preliminaries:

We denote for every direction $\xi\in\Sp^{1}:=\{x\in \R^2 : \abs{x}=1 \}$
by $\Pi_{\xi}:=\{y\in\R^2 : \langle y, \xi \rangle = 0 \}$ the
hyperplane passing through $0$ that is orthogonal to $\xi$.
If $A \subset \Omega$ is open, we denote by
$A_{\xi,y}:=\{t\in\R : y+t\xi\in A\} \subset \R$
the one-dimensional slice of $A$ indexed by $y\in\Pi_{\xi}$.
Finally, for all $w$ defined on $\Omega$, we define the one-dimensional
function $w^{\xi,y}(t)=w(y+t\xi)$ as the restriction of $w$
to $\Omega_{\xi,y}$.
\medskip

Next we define for every open set $A \subset \Omega$
a \emph{localized} functional $J_{\ve,\kappa}(u,v,A)$.
To that end, we first localize the functional $m_{\ve,\kappa}$.
We define
\[
m_{\ve,\kappa}(v,A) :=
\inf\bigl\{\mathcal{H}^0(Y) : Y \subset \R^2,\ v|_A = v_{Y,\kappa}|_A\bigr\}\;.
\]
Then we define
\[
J_{\ve,\kappa}(u,v,A):=
 \frac{1}{2}\int_A(u-f)^2 + \alpha v \lvert\nabla u\rvert^2 \,dx +
 2 \ve \beta m_{\ve,\kappa}(v,A)\;.
\]
Moreover, we define for each $\xi \in \Sp^{1}$ the directional
functional
\[
J_{\ve,\kappa}^\xi(u,v,A):=
 \frac{1}{2}\int_A(u-f)^2 + \alpha v \langle \xi, \nabla u\rangle^2 \,dx + 2\ve\beta m_{\ve,\kappa}(v,A)\;.
 \]
Finally, we consider for each $\xi\in\Sp^{1}$, $y \in \Pi_\xi$, and
$I\subset \Omega_{\xi,y}$ open the one-dimensional functionals
\[
F^{\xi,y}(\hat{u},I) =
\frac{1}{2}\int_{I\setminus S_{\hat{u}}} (\hat{u}-f^{\xi,y})^2 + (\hat{u}')^2\,dx
+ \beta\mathcal{H}^0(S_{\hat{u}})
\]
and
\[
J_{\ve,\kappa}^{\xi,y}(\hat{u},\hat{v},I) =
 \frac{1}{2}\int_I(\hat{u}-f^{\xi,y})^2 + \alpha \hat{v} (\hat{u}')^2\,dx + \beta \tilde{m}_{\ve,\kappa}(\hat{v},I)\,,
\]
where $\tilde{m}_{\ve,\kappa}$ is as defined in~\eqref{eq:tildemve}.

We claim that for every $u$, $v \in L^2(\Omega)$, $\ve > 0$, $0 < \kappa < 1$,
$A \subset \Omega$ open, and $\xi \in \Sp^{1}$ the inequalities
\begin{equation}\label{eq:Jxi}
J_{\ve,\kappa}(u,v,A) \ge J_{\ve,\kappa}^\xi(u,v,A) \ge \int_{\Pi_\xi}
J_{\ve,\kappa}^{\xi,y}(u^{\xi,y},v^{\xi,y},A_{\xi,y})\,d\mathcal{H}^{1}(y)
\end{equation}
hold. Indeed, the first inequality is a direct consequence of the
definition of the involved functionals. For the second inequality,
note first that, by Fubini's theorem,
\[
\int_A (u-f)^2 + \alpha v \langle \xi, \nabla u\rangle^2 \,dx =
\int_{\Pi^\xi} \int_{A_{\xi,y}}(u^{\xi,y}-f^{\xi,y})^2 +
\alpha v^{\xi,y} ({u^{\xi,y}})'^2\,dt\,d\mathcal{H}^{1}(y)\;.
\]
Thus it remains to show that
\begin{equation}\label{eq:mve}
2\ve\, m_{\ve,\kappa}(v,A)
\ge \int_{\Pi_\xi} \tilde{m}_{\ve,\kappa}(v^{\xi,y},A_{\xi,y})\,d\mathcal{H}^{1}(y)
\end{equation}
whenever $v \in L^2(A)$.
In case $m_{\ve,\kappa}(v,A) = +\infty$, this inequality trivially holds.
Else, there exists a set $Y = \{y_1,\ldots,y_m\} \subset \R^2$
with $m = m_{\ve,\kappa}(v,A)$ such that
\[
v(x) =
\begin{cases}
  \kappa_\ve &\text{ if } x \in \bigcup_{i=1}^m B_{\ve}(y_i)\,,\\
  1 &\text{ else.}
\end{cases}
\]
Then
\[
\begin{aligned}
2\ve\,m
&= \sum_{i=1}^m \mathcal{H}^1\bigl(\{y \in \Pi_\xi : B_{\ve}(y_i)\cap (y+\R\xi) \neq \emptyset\}\bigr)\\
&\ge \sum_{i=1}^m \mathcal{H}^1\bigl(\{y \in \Pi_\xi : B_{\ve}(y_i)\cap (y+\R\xi) \cap A\neq\emptyset\}\bigr)\\
&= \int_{\Pi_\xi}
\mathcal{H}^0\bigl(\{i : B_{\ve}(y_i)\cap (y+\R\xi) \cap A \neq \emptyset\}\bigr)
\,d\mathcal{H}^{1}(y)\;.
\end{aligned}
\]
Moreover the definition of $\tilde{m}_{\ve,\kappa}(v^{\xi,y},A_{\xi,y})$ implies that
\[
\mathcal{H}^0\bigl(\{i : B_{\ve}(y_i)\cap A_{\xi,y} \neq \emptyset\}\bigr)
\ge \tilde{m}_{\ve,\kappa}(v^{\xi,y},A_{\xi,y})
\]
for all $y$ and $\xi$. This shows~\eqref{eq:mve}, which in turn
implies~\eqref{eq:Jxi}.
\medskip

Now let $(u,v) \in L^2(\Omega)\times L^2(\Omega)$, and assume that
$\ve_j \to 0$, $u_j \to u$, and $v_j \to v$. As in the
one-dimensional case we have to show that
\[
F(u,v) \le \liminf_{j\to\infty} J_{\ve_j,\kappa(\ve_j)}(u_j,v_j)\;.
\]
Again, we assume without loss of generality that the sequence
$J_{\ve_j,\kappa(\ve_j)}(u_j,v_j)$ converges to some finite number $c < +\infty$;
the claim being trivial if $c = +\infty$. In particular, this
implies that $v = 1$ almost everywhere. Now \eqref{eq:Jxi}, Fatou's Lemma, and the
$\Gamma$-convergence result for the one-dimensional case imply that,
for each open set $A \subset \Omega$ and each direction $\xi \in
\Sp^{1}$, we have
\begin{multline*}
\liminf_{j\to\infty} J_{\ve_j,\kappa(\ve_j)}(u_j,v_j,A)\\
\begin{aligned}
 &\ge \liminf_{j\to\infty}\int_{\Pi_\xi}
J_{\ve_j,\kappa(\ve_j)}^{\xi,y}(u_j^{\xi,y},v_j^{\xi,y},A_{\xi,y})\,d\mathcal{H}^{1}(y)\\
&\ge \int_{\Pi_\xi} \liminf_{j_\to\infty}
J_{\ve_j,\kappa(\ve_j)}^{\xi,y}(u_j^{\xi,y},v_j^{\xi,y},A_{\xi,y})\,d\mathcal{H}^{1}(y)\\
&\ge \int_{\Pi_\xi} F^{\xi,y}(u^{\xi,y},A_{\xi,y})\,d\mathcal{H}^{1}(y)\\
&= \frac{1}{2}\int_{A\setminus S_u} (u-f)^2+\alpha\langle\xi,\nabla
u\rangle^2\,dx\\
&\qquad{}+ \beta\int_{\Pi_\xi} \mathcal{H}^0(S_{u^{\xi,y}}\cap A_{\xi,y})\,d\mathcal{H}^{1}(y)\\
&= \frac{1}{2}\int_{A\setminus S_u}\!\!\!\!\! (u-f)^2+\alpha\langle\xi,\nabla
u\rangle^2\,dx + \beta\int_{S_u\cap A}\!\!\!\!\! \abs{\langle \xi,\nu_u\rangle}\,d\mathcal{H}^{1}(x)\;.
\end{aligned}
\end{multline*}
Now let $(\xi_i)_{i\in\N}\subset \Sp^{1}$ be a dense subset. Then
\cite[p.~191]{Bra98} implies that
\[
\begin{aligned}
\liminf_{j\to\infty} J_{\ve_j,\kappa(\ve_j)}(u_j,v_j) &\ge
\frac{1}{2}\int_{\Omega\setminus S_u} (f-u)^2 + \alpha \sup_i \langle
\xi_i,\nabla u\rangle^2\,dx \\
&\qquad + \int_{S_u} \sup_i\abs{\langle
\xi_i,\nu_u\rangle}\,d\mathcal{H}^{n-1}(y)\\
& = F(u)\;.
\end{aligned}
\]

\subsection{The Upper Bound}

We now prove inequality~\eqref{eq:limsup}.
To that end, we consider the set $\mathcal{W}(\Omega)$
consisting of all functions $u \in \SBV(\Omega)$ for which the following hold:
\begin{enumerate}
\item $\mathcal{H}^{1}(\overline{S}_u\setminus S_u) = 0$.
\item The set $\overline{S}_u$ is the union of a finite number of
  almost disjoint line segments contained in $\Omega$.
\item $u|_{\Omega\setminus \bar{S}_u} \in W^{1,\infty}(\Omega\setminus S_u)$.
\end{enumerate}
Obviously, this set is dense in $\SBV(\Omega)$
with respect to the $L^2$-topology.
Moreover, it has been shown in~\cite{Cor97,CorToa99}
that, for every $u \in \SBV(\Omega)$, there exists a sequence $u_j \to u$
with $u_j \in \mathcal{W}(\Omega)$ for every $j$ such that $F(u_j) \to F(u)$.
Using Lemma~\ref{le:gammalimit},
for proving~\eqref{eq:limsup},
we therefore have to find for every $u \in \mathcal{W}(\Omega)$, $\delta > 0$, and $\ve_j \to 0$
sequences $u_j \to 0$, $v_j \to 1$ as $j\to\infty$, such that
$\limsup_{j\to\infty} J_{\ve_j,\kappa(\ve_j)}(u_j,v_j) \le F(u)+\delta$.

Let therefore $u \in \mathcal{W}(\Omega)$ be fixed.
By definition of $\mathcal{W}(\Omega)$,
there exist $k \in \N$ and $a_i$, $b_i \in \Omega$, $1 \le i \le k$,
such that $\overline{S}_u = \bigcup_{i=1}^k [a_i,b_i]$.
Moreover,
\[
\mathcal{H}^1(S_u) = \mathcal{H}^1(\overline{S}_u) = \sum_{i=1}^k \norm{b_i-a_i}\;.
\]

Now define for $\ve > 0$ and $0 < c < 1$
\[
K(\ve,c) := \bigl\{x \in \Omega : \dist(x,S_u) < c\ve\bigr\}\;.
\]
Let moreover $\mu(\ve,c) \in \N$ and $x_l^{(\ve,c)} \in \Omega$, $1 \le l \le \mu(\ve,c)$,
be such that
\[
K(\ve,c) \subset \bigcup_{l=1}^{\mu(\ve,c)} B_\ve(x_l^{(\ve,c)})\;.
\]
Now note that, if we place the centers of the balls on a line segment $[a_i,b_i]$,
then they cover the whole set $\bigl\{x \in \Omega : \dist(x,[a_i,b_i]) < c\ve\bigr\}$
provided that the distance between two adjacent centers is
at most $2\ve\sqrt{1-c^2}$.
Thus it follows that one can cover each set
$\bigl\{x \in \Omega : \dist(x,[a_i,b_i]) < c\ve\bigr\}$
with at most $\frac{\norm{b_i-a_i}}{2\ve\sqrt{1-c^2}} + 1$ balls of radius $\ve$.
Consequently, we can choose the centers $x_l$ in such a way that
\[
\mu(\ve,c)
\le k + \sum_{i=1}^k \frac{\norm{b_i-a_i}}{2\ve\sqrt{1-c^2}}
= k + \frac{\mathcal{H}^1(S_u)}{2\ve\sqrt{1-c^2}}\;.
\]
Let now
\[
v^{(\ve,c)}(x) :=
\begin{cases}
\kappa(\ve) &\text{if } x \in \bigcup_{l=1}^{\mu(\ve,c)} B_\ve(x_l^{(\ve,c)})\,,\\
1 & \text{else.}
\end{cases}
\]
Then, for every $c$ we have $v^{(\ve,c)} \to 1$ as $\ve \to 0$.
Moreover,
\[
m_{\ve,\kappa(\ve)}(v^{(\ve,c)},\Omega) \le \mu(\ve,c)\;.
\]

Define now
\[
u^{(\ve,c)}(x):=u(x) \, \min \left(\frac{\text{dist}(x,S_u)}{c\ve},1 \right)\;.
\]
Then $u^{(\ve,c)}(x) = u(x)$ for $x \not\in K(\ve,c)$ and $u^{(\ve,c)} \to u$ as $\ve \to 0$.
Denoting $d(x)=\dist(x,S_u)$, we have for almost every $x \in K(\ve,c)$
\[
\nabla u^{(\ve,c)}(x) = \nabla u(x) \frac{d(x)}{c\ve} + u(x) \frac{\nabla d(x)}{c\ve} \;.
\]
Thus the triangle inequality and the fact that $\abs{\nabla d(x)} = 1$
almost everywhere imply that
\[
\abs{\nabla u^{(\ve,c)}(x)}
\leq  \abs{\nabla u(x)} \frac {d(x)}{c\ve}+ \abs{u(x)} \frac{\abs{\nabla d(x)}}{c\ve}
\leq \abs{\nabla u(x)} + \frac{\norm{u}_{\infty}}{c\ve}
\]
for almost every $x \in K(\ve,c)$.
Therefore, for almost every $x \in K(\ve,c)$,
\begin{equation}
\label{uenorm2}
\abs{\nabla u^{(\ve,c)}(x)}^2 \leq 2 \norm{\nabla u}_\infty^2 + 2  \frac{\norm{u}_{\infty}^2}{c^2\ve^2} \;.
\end{equation}

Now consider each term of the functional $J_{\ve,\kappa(\ve)}$ separately.
We have
\begin{multline*}
\int_{\Omega} (u^{(\ve,c)}-f)^2 dx
= \int_{\Omega\setminus K(\ve,c)}(u-f)^2 dx
+ \int_{ K(\ve,c)}(u^{(\ve,c)}-f)^2 dx \\
\to_{\ve \to 0} \int_\Omega (u - f)^2 dx\;.
\end{multline*}
From \eqref{uenorm2} we get
\begin{multline*}
\int_{\Omega} v^{(\ve,c)} \abs{\nabla u^{(\ve,c)}}^2 dx
\le \int_{\Omega\setminus K(\ve,c)} \abs{\nabla u}^2 dx
+ \int_{K(\ve,c)} \kappa(\ve) \abs{\nabla u^{(\ve,c)}}^2 dx\\
\le \int_{\Omega\setminus K(\ve,c)} \abs{\nabla u}^2 dx
+ 2\kappa(\ve) \norm{\nabla u}_\infty^2 \mathcal{L}^2(K(\ve,c))
+ \frac{2\kappa(\ve) \norm{u}_\infty^2 \mathcal{L}^2(K(\ve,c))}{c^2\ve^2}\;.
\end{multline*}
Because $\kappa(\ve) = o(\ve)$ as $\ve \to 0$ and
\[
\mathcal{L}^2(K(\ve,c))
\le \mathcal{L}^2(K(\ve,1))
\le 2\ve\mathcal{H}^1(S_u) + k\pi\ve^2
= O(\ve)
\qquad
\text{ as } \ve \to 0\,,
\]
it follows that
\[
\limsup_{\ve \to 0} \int_{\Omega} v^{(\ve,c)} \abs{\nabla u^{(\ve,c)}}^2 dx
\le \int_{\Omega\setminus S_u} \abs{\nabla u}^2\,dx\;.
\]
Finally, the construction of $v^{(\ve,c)}$ implies that
\[
2\ve\, m_{\ve,\kappa(\ve)}(v^{(\ve,c)},\Omega)
\le 2k\ve + \frac{\mathcal{H}^1(S_u)}{\sqrt{1-c^2}}\;.
\]

Let now $\ve_j \to 0$ as $j\to \infty$ and define $u_j := u^{(\ve_j,c)}$, $v_j := v^{(\ve_j,c)}$.
Then it follows that
\[
\limsup_{j \to \infty} J_{\ve_j,\kappa(\ve_j)} (u_j,v_j)
\le \frac{1}{2}\int_\Omega (u-f)^2\,dx + \frac{\alpha}{2} \int_{\Omega\setminus S_u}\abs{\nabla u}^2\,dx
+ \frac{\beta}{\sqrt{1-c^2}}\mathcal{H}^1(S_u)\;.
\]
Since $0 < c < 1$ was arbitrary and $\mathcal{H}^1(S_u) < \infty$,
we obtain~\eqref{eq:limsup} with $\delta = (1-1/\sqrt{1-c^2})\mathcal{H}^1(S_u)$,
which concludes the proof of Theorem~\ref{theorem1}.


\section{Numerical Implementation}
\label{sec:algorithm}

\subsection{Proposed Algorithm}

Based on Theorem~\ref{thm:Gasymptotic}, we propose the following algorithm
for the approximate minimization of the functional $J_{\ve,\kappa}$
for fixed $\ve > 0$ and $\kappa > 0$.

\begin{algorithm}\label{alg}
Let $f \in L^\infty(\Omega)$, $\alpha$, $\beta > 0$, $\ve > 0$,
and $0 < \kappa < 1$ be given.

Set $k = 0$ and $K_0 := \emptyset$.
\begin{enumerate}
\item Define
  \[
  v_k(x) := \begin{cases}
    \kappa & \text{ if } x \in K_k\,,\\
    1 & \text{ if } x \in \Omega \setminus K_k\;.
  \end{cases}
  \]
\item Define
  \[
  u_k := \argmin_u G(u,v_k)\;.
  \]
\item Find $y \in \Omega\setminus K_k$ such that
  $\abs{\nabla u_k(y)}$ is maximal.
\item If
  \begin{equation}\label{eq:alg_cond}
  \dfrac{\alpha}{2}\,\pi\,\frac{1-\kappa}{1+\kappa} \abs{\nabla u_k(y)}^2 < \frac{\beta}{\ve}\,,
  \end{equation}
  stop.

  Else set $K_{k+1} := K_k \cup B_\ve(y)$,
  increase $k$ by one and go to step 1.
\end{enumerate}
\end{algorithm}

Steps 3 and 4 of the algorithm use the results of Theorem~\ref{thm:Gasymptotic}.
This theorem states that adding in the $k$-th step a point $y$ to the edge indicator
will approximately result in a decrease of the functional $G$
by approximately $\ve^2\alpha\pi(1+\kappa)\abs{\nabla u_k(y)}^2/(1-\kappa)$.
Thus, we will obtain the steepest descent, if we add a point $y$,
where $\abs{\nabla u_k(y)}$ is maximal.
At the same time the adding of another ball leads to an increase
of the term $m_{\ve,\kappa}$ by $2\beta\ve$.
In total, the value of $J_{\ve,\kappa}$ will increase if~\eqref{eq:alg_cond}
holds, else $J_{\ve,\kappa}$ will decrease and therefore it makes
sense to include the point $y$ into the edge set.

\begin{remark}
  In order to increase the performance of the algorithm,
  it makes sense to add not just one ball in each iteration,
  but rather several ones.
  Also in this case a similar approximation
  as Theorem~\ref{thm:Gasymptotic} holds, and thus the same
  criterion for adding new points can be applied.
  This strategy has been used in the numerical examples below.
\end{remark}


\subsection{Numerical Results}

We now compare the results obtained with Algorithm~\ref{alg}
with results obtained using
the approximation introduced by Ambrosio and Tortorelli~\cite{AmbTor90}.
This latter method consists in minimizing the functional
\begin{equation}\label{eq:AT}
I_\ve(u,v) :=
\frac{1}{2}\int_\Omega (u-f)^2\,dx + \frac{\alpha}{2}\int_\Omega v^2\abs{\nabla u}^2\,dx
+ \frac{1}{2}\int_\Omega \Bigl(\ve \abs{\nabla u}^2 + \frac{1}{4\ve} (v-1)^2\Bigr)\,dx\;.
\end{equation}
Again, the function $v$ serves as an edge indicator in the sense
that the points where $v$ is close to zero are an approximation
of the edge set $K$ of the solution of the Mumford--Shah functional.
In contrast to the approximation by means of the functional $J_{\ve,\kappa}$,
however, where the edge set is given as the points where the function $v$ is equal to $\kappa$,
in case of the functional~\eqref{eq:AT} one has to
threshold $v$ in order to obtain a precisely defined edge set.

The minimization of $I_\ve$ has been carried out by
alternately solving the corresponding Euler--Lagrange equations
with respect to $u$ and $v$.
For the discretization, we have used a finite element approach
with piecewise bilinear basis functions on the pixel grid.
The same discretization has been used for the computation of $u_k$
in the second step of Algorithm~\ref{alg}.

Figures~\ref{fi:boats} and~\ref{fi:peppers} show a comparison
of the results of the Ambrosio--Tortorelli approximation
and Algorithm~\ref{alg}.
The edge indicators are in both cases comparable,
though our algorithm in general classifies more points as edges.
The main difference between the results is that
the Ambrosio--Tortorelli approximation leads to a diffuse
edge indicator, while our method produces well defined edges.
As a consequence, also the smoothed images tend
to be less blurred;
compare, for instance, the various light reflections
in Figure~\ref{fi:peppers}.

\begin{figure}
\[
\begin{aligned}
&\includegraphics[width=0.46\textwidth]{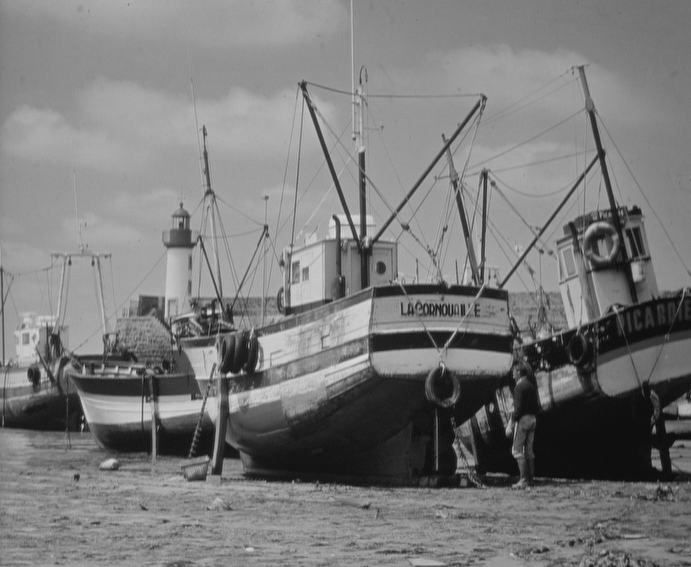} &&
&\includegraphics[width=0.46\textwidth]{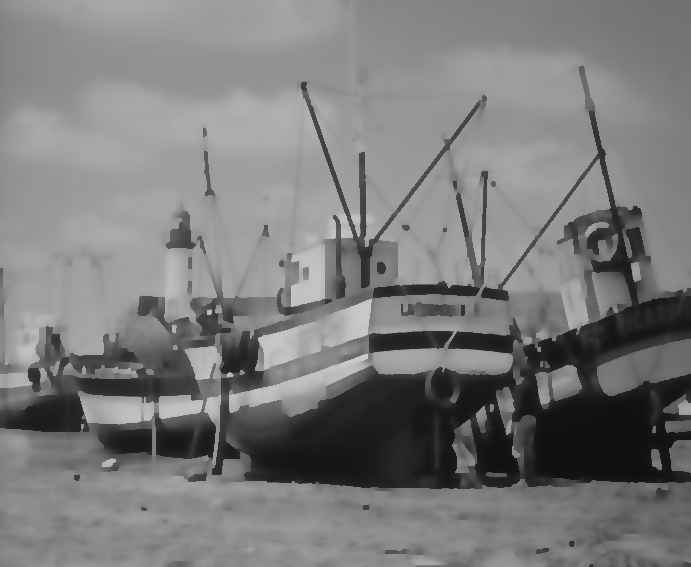} \\
&\includegraphics[width=0.46\textwidth]{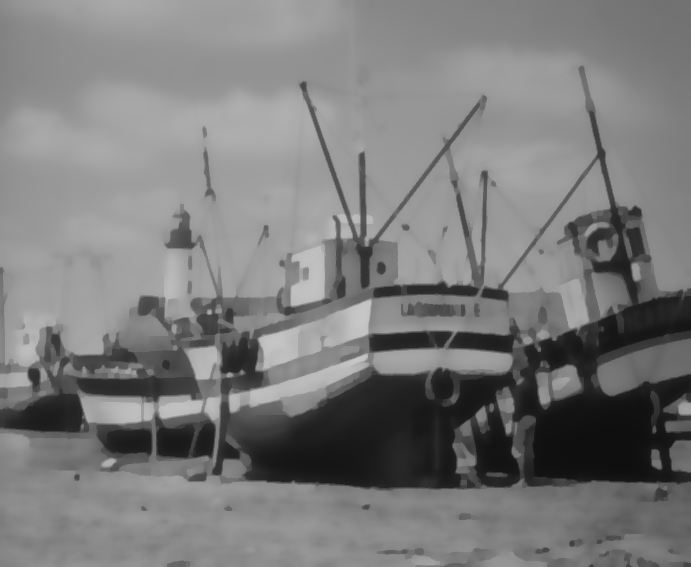} &&
&\includegraphics[width=0.46\textwidth]{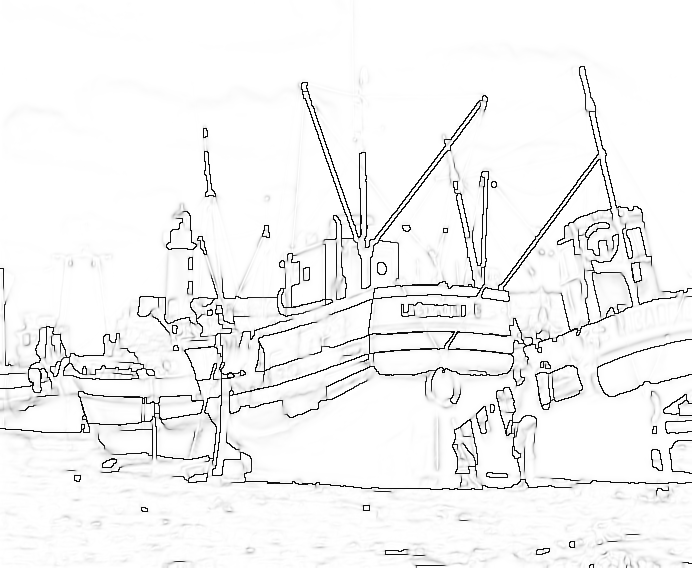} \\
&\includegraphics[width=0.46\textwidth]{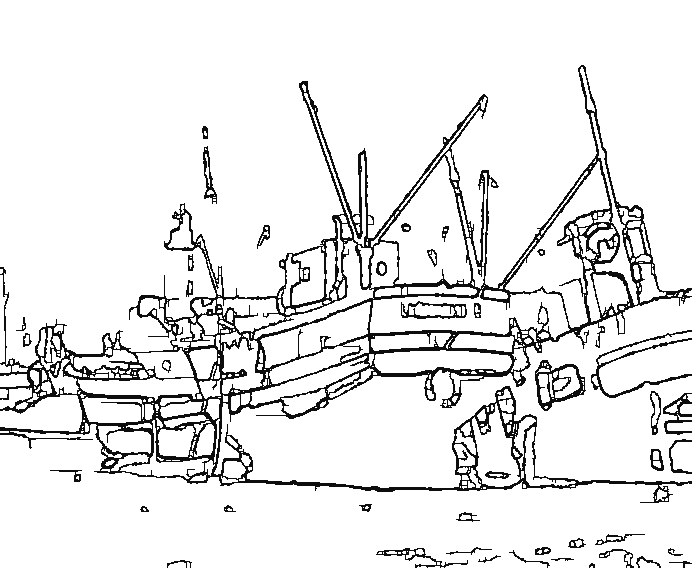} &&
&\includegraphics[width=0.46\textwidth]{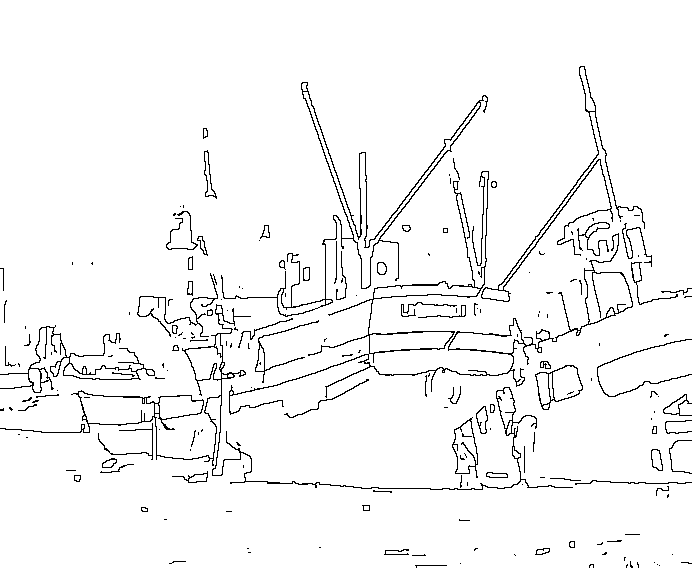}
\end{aligned}
\]
\caption{\label{fi:boats}
\emph{Upper left:} Original image.
\emph{Upper right:} Smoothing with the Ambrosio--Tortorelli Approximation.
\emph{Middle left:} Smoothing with Algorithm~\ref{alg}.
\emph{Middle right:} Edge indicator for the Ambrosio--Tortorelli Approximation.
\emph{Lower left:} Edge indicator for Algorithm~\ref{alg}.
\emph{Lower right:} Edge indicator for the Ambrosio--Tortorelli Approximation,
thresholded at 0.8.
For both methods, $\alpha = 20$, $\beta = 200$, $\ve = 0.05$.
}
\end{figure}

\begin{figure}
\[
\begin{aligned}
&\includegraphics[width=0.40\textwidth]{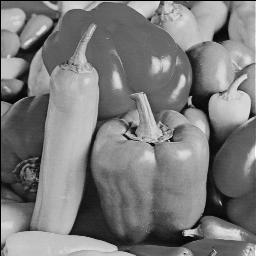} &&
&\includegraphics[width=0.40\textwidth]{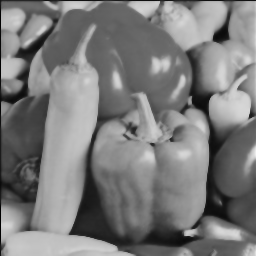} \\
&\includegraphics[width=0.40\textwidth]{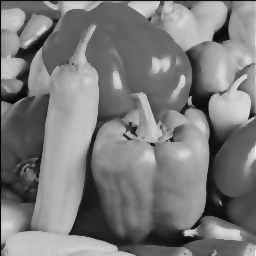} &&
&\includegraphics[width=0.40\textwidth]{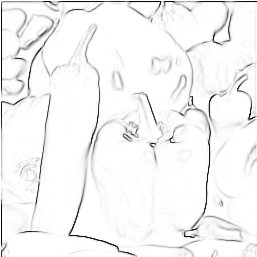} \\
&\includegraphics[width=0.40\textwidth]{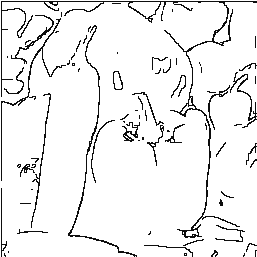} &&
&\includegraphics[width=0.40\textwidth]{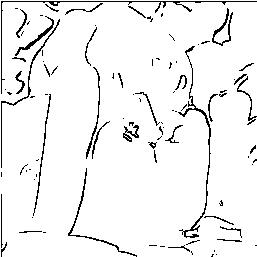}
\end{aligned}
\]
\caption{\label{fi:peppers}
\emph{Upper left:} Original image.
\emph{Upper right:} Smoothing with the Ambrosio--Tortorelli Approximation.
\emph{Middle left:} Smoothing with Algorithm~\ref{alg}.
\emph{Middle right:} Edge indicator for the Ambrosio--Tortorelli Approximation.
\emph{Lower left:} Edge indicator for Algorithm~\ref{alg}.
\emph{Lower right:} Edge indicator for the Ambrosio--Tortorelli Approximation,
thresholded at 0.9.
For both methods, $\alpha = 1$, $\beta = 500$, $\ve = 0.05$.
}
\end{figure}


\section{Conclusion}

The results of this paper provide a theoretical connection between
the Mumford--Shah functional and techniques from topological asymptotic analysis
that have recently been applied to imaging problems like edge detection.
We have shown that the Mumford--Shah functional can be approximated,
in the sense of $\Gamma$-limits, by a family of set functions that
count the number of balls that are necessary to cover the edge set
of an image. The placement of these balls can then be determined
by an asymptotic expansion of this set function with respect
to the radii of the balls.

Apart from providing yet another method for image smoothing
and segmentation, our results indicate that all the proposed algorithms
using topological asymptotic analysis are somehow related to a classical 
variational method by means of $\Gamma$-convergence. For the method based
on the function $J_{\ve,\kappa}$ defined in~\eqref{functional_Je},
the relation has been proven explicitly, but similar relations are
expected to hold for other methods. For instance, the algorithm proposed 
in~\cite{Aur09} for image segmentation should rightly be regarded as an 
implementation of the Chan--Vese model~\cite{ChaVes01} without making use
of level set methods.


\section*{Acknowledgement}

The work of OS has been supported by the Austrian Science Fund (FWF)
within the national research networks Industrial Geometry, project 9203-N12,
and Photoacoustic Imaging in Biology and Medicine, project S10505-N20.
The authors thank Prof.~Helmut Neunzert for his continuous encouragement
of this collaboration.



\end{document}